\documentclass[a4paper]{amsart}

\usepackage{amsthm}
\usepackage{amssymb}
\usepackage{cite}
\usepackage[dvipdfmx]{graphicx}
\usepackage{algorithm}
\usepackage{algorithmic} 

\theoremstyle{definition}
\newtheorem{theorem}{Theorem}[section]
\newtheorem{lemma}[theorem]{Lemma}

\theoremstyle{definition}
\newtheorem{definition}[theorem]{Definition}

\newtheorem{fact}[theorem]{Fact} 
\newtheorem{observation}[theorem]{Observation} 
\newtheorem*{ac}{Acknowledgement}

\theoremstyle{remark}
\newtheorem{remark}[theorem]{Remark}

\newenvironment{rmenum}{
\begin{enumerate}

}
{\end{enumerate}}

\newcommand{\parNei}[2]{N_{#1}(#2)}
\newcommand{\parcut}[2]{\delta_{#1}(#2)}
\renewcommand{\complement}[1]{#1^c}
\newcommand{\nonneg}{\ge 0}

\newcommand{\fcompb}[2]{\mathcal{G}(#1, #2)}

\newcommand{\constb}[2]{\mathcal{G}^+(#1, #2)}

\newcommand{\inconstb}[2]{\mathcal{G}^-(#1, #2)}

\newcommand{\inconstbh}[3]{\mathcal{G}^-_{#3}(#1, #2)}
\newcommand{\parforg}[1]{\preceq_{#1}}

\newcommand{\exth}[2]{\Sigma_{#2}(#1)}

\newcommand{\ecompb}[2]{\mathcal{C}(#1, #2)}

\newcommand{\pardm}[1]{\trianglelefteq_{#1}} 
\newcommand{\aux}[4]{{Aux}(#1; #2, #3; #4)}
\newcommand{\onevec}{1}
\newcommand{\paroforg}[1]{\preceq_{#1}^{\circ}}
\newcommand{\defib}[2]{D(#1, #2)}

\title[$b$-Matching DM Decomposition]{The Dulmage-Mendelsohn Decomposition for $b$-Matchings}
\author{Nanao Kita}
\address{National Institute of Informatics
2-1-2 Hitotsubashi, Chiyoda-ku, Tokyo, Japan 101-8430}
\email{kita@nii.ac.jp}
\date{\today}

\begin{document}
\maketitle

\begin{abstract}
We establish the theory of the Dulmage-Mendelsohn decomposition for $b$-matchings. 
The original Dulmage-Mendelsohn decomposition is a classical canonical decomposition of bipartite graphs, which describes the structures of the maximum $1$-matchings and the dual optimizers, i.e., the minimum vertex covers.  
In this paper, we develop  analogical properties,    
and thus obtain the structure of the maximum $b$-matchings and  characterizes   the family of $b$-verifying set. 
\end{abstract} 

\section{Preliminaries} 
\subsection{Notation} 
\subsubsection{General Statements} 
For standard notation of sets, graphs, and algorithms, 
we mostly refer to Schrijver~\cite{schrijver2003}. 
In the following, we list exceptions and non-standard definitions. 
Given a graph or a digraph $G$, the vertex set is denoted by $V(G)$; 
the edge set is denoted by $E(G)$ if $G$ is an (undirected) graph; otherwise, $A(G)$ denotes the arc set. 
An edge with ends $u$ and $v$ is denoted by $uv$. 
Similarly, an arc with tail $u$ and head $v$ is denoted by $uv$. 
As usual, a singleton $\{x\}$ is often denoted simply by $x$. 
We sometimes denote the vertex set of a graph $G$ simply by $G$ itself. 
For $X\subseteq V(G)$, $\complement{X}$ denotes $V(G)\setminus X$.

\subsubsection{Operations on Graphs} 
Given a graph $G$ and a set of vertices $X\subseteq V(G)$, 
the subgraph of $G$ induced by $X$ is denoted by $G[X]$. 
Given a set of edges $F$ from a supergraph of $G$,  
$G + F$ denotes the graph obtained by adding $F$ to $G$. 
Given subgraphs $H_1$ and $H_2$ of $G$, $H_1 + H_2$ denotes the union of $H_1$ and $H_2$. 

\subsubsection{Functions on Graphs} 
Given a set of vertices $X$ in a graph $G$, 
the set of neighbors of $X$ is denoted by $\parNei{G}{X}$. 
That is to say, $\parNei{G}{X}$ is the set of vertices that are adjacent to a vertex in $X$ 
and themselves are not in $X$. 
Given $X,Y\subseteq V(G)$, the set of edges whose one end is in $X$ and the other is in $Y$ is 
denoted by $E_G[X, Y]$. 
The set $E_G[X, V(G)\setminus X]$ is denoted by $\parcut{G}{X}$. 
The set of edges both of whose ends are in $X$ is denoted by $E_G[X]$. 
With respect to these functions, we often omit the subscript ``$G$'' if  it is clear from the contexts.

\subsubsection{Paths and Circuits} 
We treat paths and circuits as graphs. 
A circuit is a connected graph such that  every vertex has the degree two. 
A path is a connected graph such that  every vertex has the degree two or less 
and it is not a circuit. 
Given a path $P$ and vertices $x,y\in V(P)$, 
$xPy$ denotes the subpath between $x$ and $y$, namely, the subgraph of $P$ that is a path with ends $x$ and $y$.

\subsubsection{Ideals} 
Let $\mathcal{P}$ be a poset over a set $X$. 
Then, it is easy to observe that, for any lower ideal $I \subseteq X$, 
$X\setminus I$ is an upper ideal of $\mathcal{P}$; 
for any upper ideal $J \subseteq X$, 
$X\setminus J$ is a lower ideal of $\mathcal{P}$. 
We say a pair of an upper ideal $I$ and a lower ideal $J$ is {\em complementary} 
if $I\dot\cup J = X$. 

\subsubsection{Projective Union}
Let $G$ be a graph,  let $\mathcal{I}$ be a set of subgraphs of $G$, 
and let $W\subseteq V(G)$. 
A {\em projective union} of $\mathcal{I}$ over $W$ 
is the set $\bigcup_{H\in\mathcal{I}} V(I)\cap W$. 

\subsection{$b$-Matchings} 
Let $G$ be a graph, and let $b: V(G)\rightarrow \mathbb{Z}_{\nonneg}$. 
A set of edges $M\subseteq V(G)$ is a {\em $b$-matching} if 
$|\parcut{G}{v}\cap M| \le b(v)$ holds for each $v\in V(G)$. 
We denote $b$ by $\onevec$ if $b(v) = 1$ for each $v\in V(G)$. 
A $1$-matching is often called simply a {\em matching}.  
Given a $b$-matching $M$, 
a vetex $v\in V(G)$ is {\em $M$-loose} if $|\parcut{G}{v} \cap M | < b(v)$ holds; 
$v$ is {\em $M$-tight} if $|\parcut{G}{v} \cap M | = b(v)$. 
A {\em maximum} $b$-matching  is a $b$-matching of $G$ with the greatest number of edges. 
A $b$-matching $M$ is {\em perfect} if every vertex is $M$-tight.   

An edge is {\em $b$-allowed} if it is contained in a maximum $b$-matching; 
otherwise, it is {\em $b$-forbidden}.  
A $b$-allowed edge $e\in V(G)$ is {\em $b$-inevitable} if any maximum $b$-matching contains $e$; 
otherwise, $e$ is {\em $b$-flexible}.

A {\em $b$-elementary component} of $G$ is a subgraph $G[V(C)]$, 
where $C$ is a connected component of the subgraph of $G$ determined by 
the set of $b$-allowed edges.  
That is to say, 
a graph comprises of $b$-elementary components 
and $b$-forbidden edges that join distinct $b$-elementary components.  
The set of $b$-elementary components of $G$ is denoted by $\ecompb{G}{b}$. 

A {\em $b$-flexible component} is  a subgraph $G[V(C)]$, 
where $C$ is a connected component of the subgraph of $G$ determined by 
the set of $b$-flexible edges.  
That is to say, a graph comprises of $b$-flexible components, and $b$-inevitable and $b$-forbidden edges that join distinct $b$-flexible components. 
The set of $b$-flexible components of $G$ is denoted by $\fcompb{G}{b}$. 

As such, $b$-elementary components and $b$-flexible components can be viewed,  in their respective ways,  as fundamental building blocks of  a graph in the context of  understanding the structure of maximum $b$-matchings. Note that both concepts of $b$-elementary components and $b$-flexible components are canonical by  definition.  
Given a subgraph $C$ of $G$, define $b|_C: V(C)\rightarrow \mathbb{Z}_{\nonneg}$ as such 
that $b|_C(v) := b(v)-k_v$ for each $v\in V(C)$, 
where $k_v$ is the number of $b$-inevitable edges in $E_G[v, V(G)\setminus V(C)]$. 
It is easy to observe that 
a set of edges is a maximum $b$-matching if and only if it is the disjoint union of maximum $b|_C$-matchings of every $b$-elementary component $C$. 
Also, a set of edges is a maximum $b$-matching if and only if it is the disjoint union of the set of $b$-inevitable edges, which join distinct $b$-flexible components, and  maximum $b|_C$-matchings of every $b$-flexible component $C$; for, we will later prove the following fact: 
\begin{fact} \label{fact:allflex} 
Given a bipartite graph $G$ and a mapping $b: V(G)\rightarrow \mathbb{Z}_{\nonneg}$, 
any edge in a $b$-flexible component is $b$-flexible. 
\end{fact} 
Note that $b$-flexible components give a refinement of $b$-elementary components; 
that is, for each $C\in \ecompb{G}{b}$, there exist $D_1,\ldots, D_k\in\fcompb{G}{b}$, where $k\ge 1$, 
such that $V(C) = V(D_1)\dot\cup \cdots \dot\cup V(D_k)$. 

Let $C$ be a subgraph of $G$, which will be typically a $b$-flexible or $b$-elementary component. 
We say that $C$ is {\em trivial} if it has only one vertex. 
A vertex $v\in V(G)$ is {\em $b$-inactive} if $b(v) = 0$. 
If $v$ is an inactive vertex, then $G[v]$ forms a trivial $b$-flexible component 
as well as a trivial $b$-elementary component, 
which we say that is  {\em  $b$-inactive}.   
\begin{definition} 
Given $G$ and $b$, 
$\defib{G}{b}$ denotes the set of vertices that are $M$-loose for some maximum $b$-matchings. 
\end{definition} 
We say that $C$ of $G$ is {\em $b$-tight} (resp. {\em $b$-loose})    
if it has no vertices (resp. some vertices) in $\defib{G}{b}$. 
We say that $C$ is {\em $b$-inconsistent} 
if it is $b$-loose or if it is  $b$-inactive 
and its vertex is a neighbor of $\defib{G}{b}$; 
otherwise, $C$ is {\em $b$-consistent}. 
\begin{observation} 
Let $C\in \fcompb{G}{b}$ or $C\in \ecompb{G}{b}$, 
and let $M$ be a maximum $b$-matching of $G$. 
Then, $M\cap E(C)$ is a perfect $b|_C$-matching of $C$ if and only if $C$ is $b$-tight. 
\end{observation} 

Let $W\subseteq V(G)$, which will typically be a color class of a bipartite graph. 
 A $b$-loose subgraph $C$ is {\em hooked up by } $W$ if $V(C)\cap \defib{G}{b}\cap W \neq \emptyset$. 
A $b$-inconsistent subgraph $D$ is {\em hooked up by} $W$ if it is $b$-loose and is hooked up by $W$  or if it is  $b$-inactive and its sole vertex is 
a neighbor of $\defib{G}{b}\cap W$. %in $V(G)\setminus W$. 

The sets of $b$-consistent and $b$-inconsistent $b$-flexible components are denoted by 
$\constb{G}{b}$ and $\inconstb{G}{b}$, respectively. The sets of $b$-inconsistent $b$-flexible components hooked up by $W$ is denoted by $\inconstbh{G}{b}{W}$.

Regarding the definitions presented in this section, we will often omit the modifier ``$b$-'' if no confusion may arise. So will we for the definitions that will appear in later sections, such as $b$-verifying sets.

\section{Dulmage-Mendelsohn Decomposition for $1$-Matchings} \label{sec:dm}

In this section, 
we present the original Dulmage-Mendelsohn decomposition~\cite{dm1958, dm1959, dm1963, lovasz2009matching, murota2009matrices}, which is for $1$-matchings.  
Throughout this section, 
let $G$ be a bipartite graph with color classes $A$ and $B$.

\begin{definition} 
A set of vertices $S$ in a graph is a {\em vertex cover} if every edge has a vertex in $S$. 
\end{definition}

The maximum $1$-matching problem forms a min-max  theorem as follows: 
\begin{theorem} \label{thm:vcminmax} 
In a bipartite graph, the number of edges in a maximum $\onevec$-matchings is equal to the number of vertices in a minimum vertex cover. 
\end{theorem} 

A set of vertices is a vertex cover if and only if its complement is a stable set. 
Therefore, Theorem~\ref{thm:vcminmax} is equivalent to the following: 
\begin{theorem}\label{thm:stminmax}
In a bipartite graph, the number of edges in a maximum $\onevec$-matchings is equal to 
the value $|\complement{Z}|$, 
where $Z\subseteq V(G)$ is a maximum stable  set. 
\end{theorem}

\begin{definition} 
Let $W\in \{A, B\}$. 
Define a binary relation $\pardm{W}$ over $\ecompb{G}{\onevec}$ as follows: 
for $C_1, C_2\in \ecompb{G}{\onevec}$, $C_1\pardm{W} C_2$ holds  
if there exist $D_1,\ldots, D_k\in\ecompb{G}{\onevec}$, where $k\ge 1$,  such that 
$C_1 = D_1$, $C_2 = D_k$, and, $E[D_{i+1}\cap W, D_{i}\cap \complement{W}]\neq\emptyset$  
for each $i\in\{1,\ldots, k-1\}$. 
\end{definition}

\begin{theorem} \label{thm:dm}  
Let $G$ be a bipartite graph with color classes $A$ and $B$, 
and let $W\in \{A, B\}$.  
Then, $\pardm{W}$ is a partial order over $\ecompb{G}{\onevec}$.  
\end{theorem} 

Note that $\pardm{A}$ and $\pardm{B}$ are symmetric, that is, for any $C_1, C_2\in \ecompb{G}{b}$, $C_1 \pardm{A} C_2$ holds if and only if $C_2\pardm{B} C_2$ holds.

\begin{definition} 
A lower ideal (resp. an upper ideal) $\mathcal{I}$ of the poset $(\ecompb{G}{\onevec}, \pardm{W})$ is {\em normalized} if any inconsistent $\onevec$-elementary component hooked up by $W$ (resp. by $\complement{W}$) is in $\mathcal{I}$ 
and any inconsistent $\onevec$-elementary component hooked up by $\complement{W}$ (resp. by $W$) is disjoint from $\mathcal{I}$.  
\end{definition} 

\begin{fact}
If an inconsistent $\onevec$-elementary component is hooked up by $A$, then it is not hooked up by $B$, and vice versa. 
\end{fact}

\begin{fact} 
Any inconsistent $\onevec$-elementary component hooked up by $A$ (resp. by $B$) is minimal (resp. maximal) in the poset $(\ecompb{G}{\onevec}, \pardm{A})$. \end{fact}
Accordingly, a lower ideal (resp. an upper ideal)   $\mathcal{I}$ of the poset $(\ecompb{G}{\onevec}, \pardm{A})$ is  normalized  if and only if there exists a lower (resp. an upper) ideal $\mathcal{I}'$ such that $\mathcal{I} = \mathcal{I}' \cup \mathcal{I}_0 \setminus \mathcal{J}_0$ (resp. $\mathcal{I} = \mathcal{I}' \setminus \mathcal{I}_0 \cup \mathcal{J}_0$), where $\mathcal{I}_0$ and $\mathcal{J}_0$ are the sets of inconsistent $\onevec$-elementary component hooked up by $A$ and $B$, respectively.  

Using the concept of normalized ideals, 
the family of maximum stable sets (and accordingly, the family of minimum vertex cover as well) are characterized:  
\begin{theorem} \label{thm:st}
Let $G$ be a bipartite graph with color classes $A$ and $B$. 
A set of vertices $X\subseteq V(G)$ is a maximum stable set 
if and only if there is a complementary pair of normalized lower  and upper ideals $\mathcal{I}_A$ and $\mathcal{J}_A$ of the poset $(\ecompb{G}{\onevec}, \pardm{A})$ such that $X = X_A \dot\cup X_B$, 
where $X_A$ and $X_B$ are the projective unions of $\mathcal{I}_A$ over $A$ and of $\mathcal{I}_B$ 
over $B$, respectively. 
\end{theorem} 

Given a maximum $\onevec$-matching, 
the Dulmage-Mendelsohn decomposition can be computed in $O(n+m)$ time, where $n$ and $m$ denote the numbers of vertices and edges, respectively. 
Therefore, given $G$ and $b$, 
the Dulmage-Mendelsohn decomposition can be computed in strongly polynomial time~\cite{murota2009matrices}. 

\section{Overview of Our New Theory}

A min-max relation is known for $b$-matchings in bipartite graphs~\cite{schrijver2003, pap2005alternating}: 
\begin{theorem}\label{thm:minmax} 
Let $G$ be a bipartite graph, and let $b: V(G)\rightarrow \mathbb{Z}_{\nonneg}$. 
Then, the size of a maximum $b$-matching of $G$ 
is equal to the minimum value of $b(\complement{Z}) + E[Z]$, 
where $Z$ is taken over all subsets of $V(G)$. 
\end{theorem}

A {\em $b$-verifying set} is a set of vertices $Z$ that attains the minimum in Theorem~\ref{thm:minmax}, 
namely, whose $b(\complement{Z}) + |E[Z]|$ is equal to the size of a maximum $b$-matchings.

We establish the theory of $b$-matching analogue of the Dulmage-Mendelsohn decomposition, 
considering $b$-verifying sets and $b$-flexible components. 
The main theorems that comprise the heart of this new theory are the following: 
\begin{enumerate}
\item 
The set of flexible components form a poset with respect to a canonical binary relation, which is analogous to the poset $(\ecompb{G}{\onevec}, \pardm{W})$. (Theorem~\ref{thm:order}) 
\item 
The family of $b$-verifying set is characterized using this poset over flexible components, 
in totally analogical way to the characterization of the family of maximum stable sets. (Theorem~\ref{thm:veri}) 
\end{enumerate}  

These  structures can be computed in strongly polynomial time, which we will see in Theorem~\ref{thm:alg}. 

If we restrict ourselves to the case $b = \onevec$, the structure given by our results is generally a refinement of the classical Dulmage-Mendelsohn decomposition for $1$-matchings; note that $\onevec$-verifying sets and $\onevec$-flexible components are more fine-grained concepts than their classical counterparts, namely, than the maximum stable sets and the $\onevec$-elementary components. 
Our results for the case $b = \onevec$ totally coincide with the classical one if and only if there is no $\onevec$-inevitable edge.

\section{The Dulmage-Mendelsohn Decomposition for $b$-Matchings} \label{sec:dmb}

\subsection{Preparation} \label{sec:dmb:prep} 
Throughout Section~\ref{sec:dmb}, 
unless otherwise stated, 
let $G$ be a bipartite graph with color classes $A$ and $B$, and let $b:V(G)\rightarrow \mathbb{Z}_{\nonneg}$. 
Note that, as the roles of $A$ and $B$ are given arbitrarily, every statement  also holds by swapping $A$ and $B$.

In Section~\ref{sec:dmb:prep}, 
we present lemmas on the relationship between maximum $b$-matchings and verifying sets, 
and thus give an observation about verifying sets and  flexible components.  

The next lemma can be deduced from Theorem~\ref{thm:minmax}, however we present it with a stand-alone proof so to ensure that the whole new theory of ours  is self-contained. 
\begin{lemma}\label{lem:min2max} 
Let $G$ be a bipartite graph, and let $b: V(G)\rightarrow \mathbb{Z}_{\nonneg}$. 
Then, 
$|M| \le b(\complement{Z}) + |E[Z]|$ holds  
for any $b$-matching $M$ of $G$ and any set of vertices $Z\subseteq V(G)$. 
The equality holds if and only if 
\begin{rmenum} 
\item \label{item:min2max:inevitable} $E[Z]\subseteq M$ holds, 
\item \label{item:min2max:tight} every vertex in $\complement{Z}$ is $M$-tight, and 
\item \label{item:min2max:forbidden} $E[\complement{Z}]\cap M = \emptyset$. 
\end{rmenum} 
\end{lemma} 

\begin{proof}%lem:min2max
Obviously, $|M\cap E[Z]| \le |E[Z]|$ 
and $|M\setminus E[Z]| \le b(\complement{Z})$ hold. 
Therefore, $|M| \le  b(\complement{Z}) + |E[Z]|$ holds. 
The necessary and sufficient condition for the equality 
is easily observed by considering the above two inequality. 
\end{proof}

Lemma~\ref{lem:min2max} implies the following two lemmas:

\begin{lemma} \label{lem:opt2slack}
Let $G$ be a bipartite graph, and let $b: V(G)\rightarrow \mathbb{Z}_{\nonneg}$. 
If $Z\subseteq V(G)$ is a verifying set, then, for any maximum $b$-matching $M$, 
the following hold: 
\begin{rmenum} 
\item \label{item:opt2slack:inevitable} $E[Z]\subseteq M$;
\item \label{item:opt2slack:forbidden} $E[\complement{Z}]\cap M  = \emptyset$; 
\item \label{item:opt2slack:loose} Any $M$-loose vertex is contained in $Z$. 
\end{rmenum} 
Accordingly, any edge in $E[Z]$ is inevitable, 
whereas any edge in $E[\complement{Z}]$ is forbidden. 
\end{lemma}

\begin{lemma} \label{lem:slack2opt} 
Let $G$ be a bipartite graph, and let $b: V(G)\rightarrow \mathbb{Z}_{\nonneg}$. 
Let $M$ be a  $b$-matching. 
If a set of vertices $Z\subseteq V(G)$ satisfies 
\begin{rmenum}
\item \label{item:slack2opt:inevitable} $E[Z]\subseteq M$, 
\item \label{item:slack2opt:forbidden} $E[\complement{Z}] \cap M = \emptyset$, 
\item \label{item:slack2opt:loose} any $M$-loose vertex is contained in $Z$, 
\end{rmenum} 
then $M$ is a maximum $b$-matching and $Z$ is a verifying set. 
\end{lemma}

A set of vertices is {\em separating} if it is empty or is the union of vertex sets of some flexible components. 
We can now present a fundamental observation about relationship between verifying sets and flexible components. 

\begin{lemma} \label{lem:inconst2veri}
Let $G$ be a bipartite graph with color classes $A$ and $B$, and let $b: V(G)\rightarrow \mathbb{Z}_{\nonneg}$. 
Let $Z$ be a verifying set, and 
let $Z_1 := (Z\cap A)\dot\cup (B\setminus Z)$ and $Z_2 := (A\setminus Z) \dot\cup (Z\cap B)$. 
Then, 
\begin{rmenum} 
\item \label{item:inconst2veri:separating} $Z_1$ and $Z_2$ are separating,  
\item \label{item:inconst2veri:inconst} any $C\in \inconstbh{G}{b}{A}$ (resp. $C\in\inconstbh{G}{b}{B}$) satisfies $V(C)\subseteq Z_1$ (resp. $V(C) \subseteq Z_2$),  and 
\item \label{item:inconst2veri:disjoint} $\inconstbh{G}{b}{A}\cap \inconstbh{G}{b}{B} = \emptyset$. 
\end{rmenum} 
\end{lemma} 

\begin{proof} 
From the last statement of Lemma~\ref{lem:opt2slack}, it is obvious that \ref{item:inconst2veri:separating} follows. 
Additionally, Lemma~\ref{lem:opt2slack} \ref{item:opt2slack:loose} implies  $\defib{G}{b}\subseteq Z$. 
Hence, we have $V(C)\subseteq Z_1$ for any $C\in\inconstbh{G}{b}{A}$. 
Therefore, from Lemma~\ref{lem:opt2slack} \ref{item:opt2slack:inevitable}, 
if $v \in B$ is an inactive vertex with $\parNei{G}{v}\cap \defib{G}{b} \neq \emptyset$, 
then $v\in Z_1$ holds. 
Therefore, we obtain \ref{item:inconst2veri:inconst}. 
As $Z_1$ and $Z_2$ are disjoint, this immediately proves \ref{item:inconst2veri:disjoint}. 
\end{proof}

\subsection{Structure of Inconsistent Flexible Components} 
\subsubsection{Canonical Verifying Set} 
The goal of this section is to obtain Theorem~\ref{thm:unit}, which claims the existence of  two special verifying sets. 

Let $M$ be a $b$-matching. 
We say a path or a circuit  $M$-alternating if edges in $M$ and not in $M$ appear alternately. 
More precisely, 
a circuit $C$ is {\em $M$-alternating} 
if $|\parcut{C}{v} \cap M| = 1$ for every $v\in V(C)$. 
We define three types of $M$-alternating paths: 
A path $P$ with ends $x$ and $y$ is {\em $M$-wedge} from $x$ to $y$ if $|\parcut{P}{v} \cap M| = 1$ for each $v\in V(P)\setminus \{y\}$ whereas $\parcut{P}{y}\cap M = \emptyset$; 
A path $P$ with ends $x$ and $y$ is {\em $M$-saturated} (resp. {\em $M$-exposed}) between $x$ and $y$  if $|\parcut{P}{v} \cap M| = 1$ (resp. $|\parcut{P}{v} \cap (E(G)\setminus M)| = 1)$ for each $v\in V(P)$. 

The next one is easy to confirm: 
\begin{lemma} \label{lem:path2allowed}
Let $G$ be a bipartite graph, and let $b: V(G)\rightarrow \mathbb{Z}_{\nonneg}$. 
Let $M$ be a maximum $b$-matching,   
and let $v\in V(G)$ be an $M$-loose vertex. 
Let  $P$ be an $M$-wedge path from a vertex $u\in V(G)$ to the vertex $v$. 
 Then, $M\triangle E(P)$ is also a maximum $b$-matching of $G$. 
 Accordingly, all edges of $P$ are contained in the same flexible component. 
\end{lemma} 

The first statement and its proof of the next lemma  have been known; see Pap~\cite{pap2005alternating}, which claims that, given a maximum $b$-matching, we can construct a verifying set. The other statements prove that this verifying set is, in fact, canonical, according to its relationship with inconsistent flexible components: 
\begin{lemma} \label{lem:path2veri} 
Let $G$ be a bipartite graph with color classes $A$ and $B$,  
let $b: V(G)\rightarrow \mathbb{Z}_{\nonneg}$, 
and let $M$ be a maximum $b$-matching of $G$. 
Let $U_A$ be the set of $M$-loose vertices in $A$. 
Let $S_A\subseteq A$ be the set of vertices from which to vertices in $U_A$   
there exist  $M$-wedge paths, 
and let $T_A \subseteq B$ be the set of vertices between which and vertices in $U_A$ 
there exit  $M$-exposed paths. 
Then, 
\begin{rmenum} 
\item \label{item:path2veri:veri} $S_A\cup (B\setminus T_A)$ is a verifying set, and 
\item \label{item:path2veri:d} $S_A = \defib{G}{b}\cap A$. 
\end{rmenum} 
Accordingly, if $Z_1 := S_A \dot\cup T_A$, 
then $Z_1$ is the disjoint union of vertex sets of all inconsistent flexible components hooked up by $A$. 
\end{lemma} 

\begin{proof}%lem:path2veri
First we prove $E[S_A, B\setminus T_A]\subseteq M$ and $E[A\setminus S_A, T_A]\cap M = \emptyset$. 
Suppose there is an edge $xy \in E(G)\setminus M$ with  $x\in S_A$ and $y\in B \setminus T_A$.   
By definition, there is an $M$-wedge path $P$ from $x$ to a vertex $z\in U_A$. 
Then, $P+xy$ is an $M$-exposed path between $z$ and $y$, which is a contradiction. 
Next suppose there is an edge $xy\in M$ with $x\in A\setminus S_A$ and $y\in T_A$. 
By definition, there is an $M$-exposed path $P$ between a vertex $z\in U_A$ and $y$. 
Then, $P+xy$ is an $M$-wedge path from $x$ to $z$, which is again a contradiction. 

Next we prove that all vertices in $T_A$ are $M$-tight.  
Suppose a vertex $y\in T_A$ is $M$-loose, 
and let $P$ be an $M$-exposed path between some vertex in $U_A$ and $y$. 
Then, $M\triangle E(P)$ is a $b$-matching of $G$ that is larger than $M$, which is a contradiction. 
Therefore, $Z_1$ contains all $M$-loose vertices, and hence, 
from Lemma~\ref{lem:slack2opt}, the statement \ref{item:path2veri:veri} is proved. 

For the statement \ref{item:path2veri:d}, 
we first prove $S_A\subseteq \defib{G}{b}\cap A$. 
By definition, for any $x\in S_A$, 
there is an $M$-wedge path from $x$ to a vertex $z\in U_A$. 
Then, $M\triangle E(P)$ is  a maximum $b$-matching in which $x$ is $M\triangle E(P)$-loose. 
Therefore, $x\in \defib{G}{b}$ holds, and we have $S_A\subseteq \defib{G}{b}\cap A$. 
On the other hand, 
according to  \ref{item:path2veri:veri} and Lemma~\ref{lem:opt2slack} \ref{item:opt2slack:loose}, 
any vertex in $A\setminus S_A$ is $M'$-tight with respect to any maximum $b$-matching $M'$. 
Therefore, we have $S_A\supseteq \defib{G}{b}\cap A$. 
Thus, \ref{item:path2veri:d} is proved. 

From \ref{item:path2veri:veri} and Lemma~\ref{lem:inconst2veri} \ref{item:inconst2veri:inconst}, 
it follows that $S_A\dot\cup T_A$ is a separating set. 
As $S_A = \defib{G}{b}\cap A$, 
any loose flexible component $C$ hooked up by $A$ satisfies $V(C)\subseteq S_A\dot\cup T_A$; 
conversely, if $C\in\fcompb{G}{b}$ with $V(C)\cap A \neq \emptyset$ satisfies $V(C)\subseteq S_A\dot\cup T_A$, 
then $C$ is a loose flexible component hooked up by $A$. 
Therefore, $S_A \dot\cup T_A$ consists of the vertex sets of all loose flexible components hooked up by $A$ 
and some  trivial flexible components with their sole vertices in  $T_A$;
we prove in the following that those trivial are exactly the inactive flexible components hooked up by $A$.   

Let $v\in T_A$ be such that $G[v]$ is a trivial flexible component.  
By the definition of $T_A$, there is an $M$-exposed path $P$ between $v$ and  an $M$-loose vertex $w\in A$. 
If $b(v) > 0$ holds, then from Lemma~\ref{lem:opt2slack} \ref{item:opt2slack:loose}, 
there exists a vertex $u\in A$ with $uv\in M$. 
Then, $P  + uv$ is an $M$-wedge path from $u$ to $w$. 
From Lemma~\ref{lem:path2allowed}, this implies that $v$ and $w$ are in the same flexible component. 
This is a contradiction. 
Hence, $v$ is an inactive vertex. 
Moreover, if $z\in V(P)$ is such that $zv\in E(P)$, 
then of course $zv\not\in M$ holds, 
and, as $wPz$ is an $M$-wedge path from $z$ to $w$, we have $z\in S_A$; 
namely, $\parNei{G}{v} \cap S_A \neq \emptyset$. 

Conversely, if $v\in B$ is an inactive vertex with $\parNei{G}{v}\cap S_A \neq \emptyset$, then, from Lemma~\ref{lem:inconst2veri} \ref{item:inconst2veri:inconst}, $v\in T_A$ holds. 
Therefore, $S_A\dot\cup T_A$ is the union of vertex sets of all inconsistent flexible components hooked by $A$. 

\end{proof}

\begin{remark} 
Note that Lemmas~\ref{lem:min2max} and \ref{lem:path2veri} provide the proof of  Theorem~\ref{thm:minmax}. 
\end{remark} 

\begin{definition} 
Let $W \in \{A, B\}$. 
The {\em inconsistent unit} of $G$ {\em hooked up by} $W$ is the  disjoint union of vertex sets of inconsistent flexible component hooked up by $W$, 
and is denoted by $\exth{G}{W}$. 
\end{definition} 

Lemma~\ref{lem:path2veri} proved that  
$\exth{G}{A}$ and $\exth{G}{B}$ determine  verifying sets that are canonical. We can further observe, from Lemma~\ref{lem:inconst2veri},  that what they give are the ``nucleus'' of every verifying set: 

\begin{theorem} \label{thm:unit} 
Let $G$ be a bipartite graph with color classes $A$ and $B$, and 
let $b: V(G)\rightarrow \mathbb{Z}_{\nonneg}$. 
\begin{rmenum} 
\item Then, $(\exth{G}{A} \cap A) \dot\cup (B \setminus \exth{G}{B})$ 
and $(\exth{G}{B}\cap B)\dot\cup (A\setminus \exth{G}{A})$ are verifying sets. 
\item 
Any verifying set $Z$ contains $(\exth{G}{A}\cap A)\dot\cup (\exth{G}{B}\cap B)$ 
and is disjoint from $(\exth{G}{A}\cap B) \dot\cup (\exth{G}{B}\cap A)$.  
\end{rmenum} 
\end{theorem}

\subsubsection{Inner Structure of Inconsistent Unit}In this section, we investigate relationships between inconsistent flexible components. 
The results here will later  utilized in Section~\ref{sec:dmb:all} to prove that inconsistent flexible components are minimal or maximal in the poset formed by the set of flexible components, 
or in Section~\ref{sec:alg} to construct an algorithm. 

\begin{lemma} \label{lem:inconst2path} 
Let $G$ be a bipartite graph with color classes $A$ and $B$, and 
let $b: V(G)\rightarrow \mathbb{Z}_{\nonneg}$. 
Let $C$ be a loose flexible component hooked up by $A$. 
Denote $A_C := V(C)\cap A$ and $B_C := V(C)\cap B$.  
Let $M$ be a maximum $b$-matching of $G$, and let $U_A\subseteq A$ be the set of $M$-loose vertices in $A$. 
Then, 
\begin{rmenum} 
\item \label{item:inconst2path:a}
$A_C \cap U_A \neq \emptyset$, and 
$A_C$ is the set of vertices from which to  $M$-loose vertices in $A_C$ 
there exist  $M$-wedge paths of $C$;  
and,  
\item \label{item:inconst2path:b}
$B_C$ is the set of vertices in $B$ 
that are contained in $M$-wedge paths from vertices in $A_C$ to vertices in $U_A$.  
\end{rmenum} 
Accordingly, 
for any vertex $v\in B_C$, 
$C$ has an $M$-exposed path between $v$ and a vertex in $U_A$, 
and an $M$-saturated path between $v$ and a vertex in $A_C$.  
\end{lemma} 

\begin{proof} 
From Lemma~\ref{lem:path2veri}, for any $v\in A_C$, there is an $M$-wedge path $P$ from $v$ 
to an $M$-loose vertex $u\in A$. 
From Lemma~\ref{lem:path2allowed}, $P$ is a path of $C$ and thus $u\in V(C)$ follows.  
Therefore, $A_C$ has some vertices in $U_A$, 
and for any $v\in A_C$,  $C$ has an $M$-wedge path from $v$ to a vertex in $U_A$. 
The converse direction of \ref{item:inconst2path:a} is obvious. 
Hence, \ref{item:inconst2path:a} is proved. 

As $B_C \subseteq \exth{G}{A}\cap B$, Lemma~\ref{lem:path2veri} implies that, 
for any vertex $v\in B_C$,  
there is an $M$-exposed path $Q$ between $v$ and a vertex in $U_A$. 
As $A_C\neq\emptyset$, we have $b(v) > 0$  for any $v\in B_C$. 
Hence, there is a vertex $w\in A$ with $wv\in M$, 
and $Q + vw$ is an $M$-wedge path from $w$ to $v$. 
From Lemma~\ref{lem:path2allowed}, this path $Q + vw$ is contained in $C$, 
and accordingly so is $Q$. 
As the converse direction of \ref{item:inconst2path:b} is obvious, 
this proves \ref{item:inconst2path:b}. 
From \ref{item:inconst2path:b}, the remaining statement follows. 
\end{proof}

 Lemma~\ref{lem:inconst2path} derives the next lemma. Lemma~\ref{lem:lcomp2lcomp} will imply in Section~\ref{sec:dmb:all} that any  two  distinct loose flexible components hooked up by the same color class are not compatible in the poset.

\begin{lemma} \label{lem:lcomp2lcomp} 
Let $G$ be a bipartite graph with color classes $A$ and $B$, and 
let $b: V(G)\rightarrow \mathbb{Z}_{\nonneg}$. 
Let $C_1$ and $C_2$  be two distinct loose flexible components hooked up by $A$. 
Then, $E[C_1, C_2] = \emptyset$. 
\end{lemma} 

\begin{proof} 
Suppose to the contrary, namely, that 
there exist $u\in V(C_1)$ and $v\in V(C_2)$ with $uv\in E(G)$. 
Without loss of generality, assume $u\in A$ and $v\in B$. 
Let $M$ be an arbitrary maximum $b$-matching of $G$. 
First consider the case with $uv\not\in M$. 

From Lemma~\ref{lem:inconst2path}, 
$C_1$ has an $M$-wedge path $P_1$ from $u$ to an $M$-loose vertex $w$ in $A\cap V(C_1)$. 
According to the last statement of Lemma~\ref{lem:inconst2path}, 
$C_2$ has an $M$-saturated path $P_2$ between $v$ and a vertex $z\in A\cap V(C_2)$. 
Then, $P_1 + uv + P_2$ is an $M$-wedge path from $z$ to $w$, 
which implies, from Lemma~\ref{lem:path2allowed}, that $z$ and $w$ are in the same flexible component. 
Hence, we reach a contradiction for this case. 

Second, consider the case with $uv\in M$. 
According to the last statement of Lemma~\ref{lem:inconst2path}, 
$C_2$ has an $M$-exposed path $Q$ between $v$ and an $M$-loose vertex  $x\in A\cap V(C_2)$. 
Then, $uv +  Q$ is an $M$-wedge path from $u$ to $x$, 
which implies, from Lemma~\ref{lem:path2allowed}, that $u$ and $x$ are in the same flexible component. 
Thus, we again reach a contradiction, and this completes the proof. 
\end{proof}

From Lemma~\ref{lem:opt2slack} and Theorem~\ref{thm:unit}, the next lemma is obtained.    

\begin{lemma} \label{lem:inconst2const} 
Let $G$ be a bipartite graph with color classes $A$ and $B$, and 
let $b: V(G)\rightarrow \mathbb{Z}_{\nonneg}$. 
Let $C$ be an inconsistent flexible component hooked up by $A$. 
Then, for any $D\in\fcompb{G}{b}\setminus \inconstbh{G}{b}{A}$, 
the edges in $E[V(C)\cap A, V(D)\cap B]$ are inevitable, 
whereas the edges in $E[V(D)\cap A, V(C)\cap B]$ are forbidden. 
\end{lemma}

From Lemma~\ref{lem:lcomp2lcomp}, the next lemma, which will be used in Section~\ref{sec:alg},  is also derived easily:  

\begin{lemma} \label{lem:inconst2conn} 
Let $G$ be a bipartite graph with color classes $A$ and $B$, and 
let $b: V(G)\rightarrow \mathbb{Z}_{\nonneg}$. 
Let $S$ be the set of inactive vertices in $\exth{G}{A}$. 
Then, the set of loose flexible components hooked up by $A$
  is equal to the set of connected components of $G[S]$. 
\end{lemma}

\subsection{Structure of Consistent Flexible Components} 
In this section, we obtain Theorem~\ref{thm:order}, which states that the set of consistent flexible components forms a poset with respect to a certain canonical partially order we define here. 

A graph is {\em $b$-flexible connected} if it has only one flexible component.

\begin{lemma} \label{lem:elempath} 
Let $G$ be a bipartite graph with color classes $A$ and $B$, 
and let $b: V(G)\rightarrow \mathbb{Z}_{\nonneg}$. 
If $G$ is a flexible connected graph with a perfect $b$-matching $M$, 
\begin{rmenum} 
\item \label{item:elempath:wedge} 
then, for any $x\in A$ and any $y\in B$, there is an $M$-wedge path from $x$ to $y$; 
\item \label{item:elempath:saturated} for any $x\in A$ and any $y\in B$, there is an $M$-saturated path between $x$ and $y$; and, 
\item \label{item:elempath:exposed} for any $x\in A$ and any $y\in B$, there is an $M$-exposed path between $x$ and $y$. 
\end{rmenum} 
\end{lemma}

\begin{proof} 
Let $x\in A$ be an arbitrary vertex.  
Let $S$ be the set of vertices in $A$ from $x$ to which 
there is an $M$-wedge path, 
and let $T\subseteq B$ be the set of vertices between which and $x$ 
there is an $M$-saturated  path. 
Then, by this definition, the edges in $E[S, B\setminus T]$ are disjoint from  $M$, 
whereas the edges in $E[A\setminus S, T]$ are in $M$.  
Let $Z := (A\setminus S) \dot\cup T$. 
As $G$ has no $M$-loose vertices, 
$Z$ is a verifying set, according to Lemma~\ref{lem:slack2opt}. 
Hence, if neither $S\dot\cup T = \emptyset$ nor $S\dot\cup T = V(G)$ holds, 
then    Lemma~\ref{lem:inconst2veri}  \ref{item:inconst2veri:separating} implies that $G$ is not flexible connected, 
which is a contradiction. 
Obviously, $S\dot\cup T \neq \emptyset$. Therefore, we have $S \dot\cup T = V(G)$, 
that is to say, $S = A$ and $T = B$. 
Thus, \ref{item:elempath:wedge} and \ref{item:elempath:saturated} are proved. 

To prove \ref{item:elempath:exposed},  
let $S'$ be the set of vertices from which to $x$ there is an $M$-wedge path, 
and let $T'$ be the set of vertices between which and $x$ there is an $M$-exposed path. 
From the symmetric argument similar to the above, 
we can prove \ref{item:elempath:exposed}. 
\end{proof}

The next lemma is easy to confirm: 

\begin{lemma} \label{lem:restrict} 
Let $G$ be a bipartite graph, and 
let $b: V(G)\rightarrow \mathbb{Z}_{\nonneg}$. 
If  $C\in\constb{G}{b}$ holds,  then $C$ is a $b|_C$-flexible connected graph 
with a perfect $b|_C$-matching. 
\end{lemma} 

\begin{remark} 
Under Lemma~\ref{lem:restrict}, 
the claim of Lemma~\ref{lem:elempath} can be applied for each consistent flexible component by treating it as a flexible connected graph. 
\end{remark}

We can now prove Fact~\ref{fact:allflex}. 
\begin{proof}[Proof of Fact~\ref{fact:allflex}]
Let $G$ be a bipartite graph with color classes $A$ and $B$, and let  $b: V(G)\rightarrow \mathbb{Z}_{\nonneg}$. Let $C\in\fcompb{G}{b}$, and let $uv\in E(C)$, where $u\in A$ and $v\in B$. 

First consider the case with $C\in\constb{G}{b}$. Let $M$ be an arbitrary maximum $b$-matching of $G$. 
Under Lemma~\ref{lem:elempath}, 
if $uv\in M$ holds, then let $P$ be an $M$-exposed path between $u$ and $v$; 
otherwise, let $P$ be an $M$-saturated path between $u$ and $v$. 
Then, $P + uv$ is an $M$-alternating circuit, 
and hence, $M\triangle E(P+uv)$ is also a maximum $b$-matching. 
The edge $uv$ is exclusively contained in either 
$M$ or $M \triangle E(P+uv)$, and therefore, $uv$ is a flexible edge. 

Next consider the case with $C\in\inconstb{G}{b}$. 
It suffices to consider the case where $C$ is a loose flexible component hooked up by $A$. 
According to Lemma~\ref{lem:path2veri}, 
there is a maximum $b$-matching $M$ such that $u$ is $M$-loose. 
Under Lemma~\ref{lem:inconst2path}, if $uv\in M$ holds, then let $P$ be an $M$-exposed path between $v$ and an $M$-loose vertex in $V(C)\cap A$; 
otherwise, let $P$ be an $M$-saturated path between $v$ and a vertex in $V(C)\cap A$. 
Then, $P + uv$ is an $M$-alternating circuit or an $M$-wedge path from an $M$-loose vertex to a vertex in $V(C)\cap A$. 
Hence, $M$ and $M\triangle E(P + uv)$ are both maximum $b$-matchings, exactly  one of which contains $uv$. 
Therefore, again $uv$ is a flexible edge. 
\end{proof}

In the following, we first define a binary relation $\paroforg{W}$ over $\fcompb{G}{b}$, 
and then, using this,  further define a binary relation $\parforg{W}$ over $\constb{G}{b}$, where $W\in \{A, B\}$; this $\parforg{W}$ will turn out, in Theorem~\ref{thm:order}, to be a partial order. 

\begin{definition} 
Let $W\in \{A, B\}$. 
We define a binary relation $\paroforg{W}$ over $\fcompb{G}{b}$ as follows:
For $C_1, C_2\in\fcompb{G}{b}$, 
$C_1\paroforg{W} C_2$ holds if $C_1 = C_2$ or if there is an inevitable edge between $V(C_1)\cap W$ and $V(C_2)\cap \complement{W}$, or is a forbidden edge between $V(C_2)\cap W$ and $V(C_1)\cap \complement{W}$. 
Furthermore, we define a binary relation $\parforg{W}$ over $\constb{G}{b}$ as follows: 
For $C_1, C_2\in\constb{G}{b}$, 
$C_1\parforg{W} C_2$ holds if 
there exist $D_1,\ldots, D_k\in\constb{G}{b}$ with $k\ge 1$ such that, 
$D_1 = C_1$, $D_2 = C_2$, and $D_1\paroforg{W} \cdots \paroforg{W} D_k$ hold. 
\end{definition} 

\begin{remark} 
The two binary relations $\parforg{A}$ and $\parforg{B}$ are symmetric. 
That is, $C_1\parforg{A} C_2$ holds if and only if $C_2\parforg{B} C_1$ holds. 
\end{remark} 

From Lemmas~\ref{lem:elempath} and \ref{lem:restrict}, we derive the following lemma: 
\begin{lemma} \label{lem:order2path} 
Let $G$ be a bipartite graph with color classes $A$ and $B$, and 
let $b: V(G)\rightarrow \mathbb{Z}_{\nonneg}$. 
Let $C_1$ and $C_2$ be consistent flexible components with $C_1\parforg{A} C_2$. 
Let $D_1,\ldots, D_k\in\constb{G}{b}$, where $k\ge 1$, 
be such that $C_1 = D_1$, $C_2 = D_k$, 
and $D_1\paroforg{A} \cdots \paroforg{A} D_k$. 
Let $A_i := V(C_i)\cap A$ and $B_i := V(C_i)\cap B$ for each $i\in \{1,2\}$. 
Then, for any maximum $b$-matching $M$ of $G$, the following hold: 
\begin{rmenum} 
\item \label{item:order2path:aa} For any $x\in A_1$ and any $y\in A_2$, 
there is an $M$-wedge path from $x$ to $y$; 
\item \label{item:order2path:ab} For any $x\in A_1$ and any $y\in B_2$, 
there is an $M$-saturated path between $x$ and $y$; 
\item For any $x\in A_2$ and any $y\in B_1$, 
there is an $M$-exposed path between $x$ and $y$; and, 
\item \label{item:order2path:bb} 
For any $x\in B_1$ and any $y\in B_2$, 
there is an $M$-wedge path from $y$ to $x$. 
\end{rmenum} 
Additionally, these paths can be taken so that 
their vertices are contained in $V(D_1)\dot\cup \cdots \dot\cup V(D_k)$. 
\end{lemma}

\begin{proof}  
We proceed by induction on $k$. 
If $k=1$, then, from Lemmas~\ref{lem:elempath} and \ref{lem:restrict}, the statements hold. 
Let $k > 1$, and assume the statements hold for the cases where the parameter is less. 
Note that under this hypothesis, the statements hold for $D_1$ and $D_{k-1}$, 
 which satisfy $D_1 \paroforg{A}\cdots \paroforg{A} D_{k-1}$. 
If $D_k$ is equal to one of $D_1,\ldots, D_{k-1}$, 
then, of course, we are done. Hence, in the following, assume $D_k \neq D_i$ for any $i\in\{1, \ldots, k-1\}$.  
Let $x$ be an arbitrary vertex from $V(D_1)$, 
and let $y$ be an arbitrary vertex from $V(D_{k})$. 

First consider the case where there exists $uv\in M$ 
with $u\in V(D_{k-1})\cap A$ and $v \in V(D_k)\cap B$. 
If $x\in A$ holds, then let $P$ be an $M$-wedge path $P$ from $x$ to $u$; 
otherwise, let $P$ be an $M$-exposed path between $u$ and $x$. 
From the  hypothesis, we can take $P$ so that  $V(P)\subseteq V(D_1)\dot\cup\cdots \dot\cup V(D_{k-1})$ holds. 
On the other hand, under Lemmas~\ref{lem:elempath} and \ref{lem:restrict}, 
if $y\in A$ holds, then let $Q$ be an $M$-exposed path of $D_k$ between $y$ and $v$; 
otherwise, let $Q$ be an $M$-wedge path of $D_k$ from $y$ to $v$. 
Then, $P + uv + Q$ is a path with $V(P + uv + Q) \subseteq V(D_1)\dot\cup\cdots \dot\cup V(D_k)$, 
which is  $M$-wedge  from $x$ to $y$, 
 $M$-saturated  between $x$ and $y$,  
 $M$-exposed  between $y$ and $x$, 
or  $M$-wedge  from $y$ to $x$,  
according to the cases  
with $x\in A$ and $y\in A$, with $x\in A$ and $y\in B$,  with $x\in B$ and $y\in A$, 
or with $x\in Y$ and $y\in Y$,  respectively. 
Thus, the statements hold for $D_1$ and $D_k$ 
for this case. 

Next consider the case where there exists $uv\in E(G)\setminus M$ 
with $u\in V(D_{k-1})\cap B$ and $v\in V(D_k)\cap B$. 
In this case, the statements are also proved to hold for $D_1$ and $D_k$, 
in the similar way as the above. 
This completes the proof. 
\end{proof}

From Lemma~\ref{lem:order2path}, 
we can now prove that $(\constb{G}{b}, \parforg{W})$ forms a poset for each $W\in \{A, B\}$: 
\begin{theorem} \label{thm:order} 
Let $G$ be a bipartite graph with color classes $A$ and $B$, and 
let $b: V(G)\rightarrow \mathbb{Z}_{\nonneg}$. 
Then, the binary relation $\parforg{A}$ is a partial order over $\constb{G}{b}$. 
\end{theorem} 

\begin{proof} 
As reflexibity and transitivity are obvious from the definition, 
we prove antisymmetry in the following. 
Let $C_1, C_2\in\constb{G}{b}$ be such that $C_1\parforg{A} C_2$ and $C_2\parforg{A} C_1$. 
Let $D_1,\ldots, D_k\in\constb{G}{b}$, where $k \ge 1$,  be such that $C_1 = D_1$, $C_2 = D_k$, 
and $D_1\paroforg{A} \cdots \paroforg{A} D_k$.  
Let $D_{k}, \ldots, D_l\in\constb{G}{b}$, where $l\ge k$ be such that 
$D_l = C_1$ and $D_k \paroforg{A} \cdots \paroforg{A} D_l$. 
Suppose antisymmetry fails, that is, suppose $C_1\neq C_2$. 
Then, we can suppose $l \ge 3$. 
Without loss of generality, 
we can assume $D_i \neq D_{i+1}$ for each $i\in \{1, \ldots, l-1\}$. 
Let $p \in \{3,\ldots, l\}$ be the smallest number with $D_p \in \{D_1, \ldots, D_{p-1}\}$. 
Let $q \in \{1,\ldots, p-1\}$ be such that $D_p = D_q$. 
Note that $D_q\neq D_{q-1}$ and $D_q \paroforg{A} \cdots \paroforg{A} D_{p-1} \paroforg{A} D_p$ hold.

First consider the case where $D_{p-1} \paroforg{A} D_p$ is given by an edge 
$uv\in M$ with $u\in V(D_{p-1})\cap A$ and $v\in V(D_p)\cap B$. 
From Lemma~\ref{lem:order2path}, as $D_q \paroforg{A} \cdots \paroforg{A} D_{p-1}$ holds, 
there is an $M$-exposed path $P$ between $u$ and $v$. 
Then, $P + uv$ is an $M$-alternating circuit that shares some vertices with more than one flexible component. 
Therefore, $M\triangle E(P+uv)$ is a maximum $b$-matching of $G$, 
which excludes some inevitable edges or contains some forbidden edges. 
This is a contradiction. 

Next consider the case where there is an edge $uv\in E(G)\setminus M$ 
with $u\in V(D_{p-1})\cap B$ and $v\in V(D_{p})\cap A$. 
In this case, take $P$ as an $M$-saturated path between $u$ and $v$. 
Then, $P + uv$ is again an $M$-alternating circuit, 
and in the same way, we reach a contradiction. 
Therefore, we obtain $C_1 = C_2$, and this completes the proof. 
\end{proof}

\subsection{Extension over All Flexible Components} \label{sec:dmb:all} 

In this section, we prove, in Theorem~\ref{thm:extorder},  the canonical partially ordered structure over the set of all flexible components. 
From Lemma~\ref{lem:inconst2const}, 
if $D_1,\ldots, D_k\in \fcompb{G}{b}$, where $k\ge 1$, with $D_1 \paroforg{A} \cdots \paroforg{A} D_k$ satisfy $D_1, D_2\in \constb{G}{b}$, 
then $D_i\in \constb{G}{b}$ holds for each $i\in \{1,\ldots, k\}$. 
Therefore, the definition of $\parforg{W}$ can be compatibly extended over $\fcompb{G}{b}$ as follows. 

\begin{definition} 
Let $W\in \{A, B\}$. 
We define a binary relation $\parforg{W}$ over $\fcompb{G}{b}$ as follows: 
For $C_1, C_2\in\fcompb{G}{b}$, 
$C_1\parforg{W} C_2$ holds if  
there exist $D_1,\ldots, D_k\in\fcompb{G}{b}$, where $k \ge 1$, with 
$D_1 \paroforg{W} \cdots \paroforg{W} D_k$. 

\end{definition} 

The next lemma follows from Lemmas~\ref{lem:lcomp2lcomp} and \ref{lem:inconst2const}. 

\begin{lemma} \label{lem:inconstorder}
Let $G$ be a bipartite graph with color classes $A$ and $B$, and 
let $b: V(G)\rightarrow \mathbb{Z}_{\nonneg}$. 
If $C\in\inconstbh{G}{b}{A}$ satisfies $D\parforg{A} C$ for $D\in \fcompb{G}{b}\setminus \{C\}$, 
then $D$ is an inactive flexible component hooked up by $A$ 
and $C$ is a loose flexible component hooked up by $A$ with $\parNei{G}{D}\cap V(C) \neq\emptyset$. 
If $C\in\inconstbh{G}{b}{B}$ satisfies $C \parforg{B} D$ for $D\in \fcompb{G}{b}\setminus \{C\}$, 
then $D$ is an inactive flexible component hooked up by $B$ 
and $C$ is a loose flexible component hooked up by $B$ with $\parNei{G}{D}\cap V(C) \neq\emptyset$.   
\end{lemma}

From Lemma~\ref{lem:inconstorder} and Theorem~\ref{thm:order}, we derive Theorem~\ref{thm:extorder}. 
\begin{theorem} \label{thm:extorder} 
Let $G$ be a bipartite graph with color classes $A$ and $B$, and 
let $b: V(G)\rightarrow \mathbb{Z}_{\nonneg}$. 
Then, the binary relation $\parforg{A}$ is a partial order over $\fcompb{G}{b}$. 
\end{theorem} 
\begin{proof} 
Reflexivity and transitivity are obvious from the definition, 
hence we prove antisymmetry in the following. 
Let $C_1, C_2\in\constb{G}{b}$ be such that $C_1\parforg{A} C_2$ and $C_2\parforg{A} C_1$. 
Let $D_1,\ldots, D_k\in\constb{G}{b}$, where $k \ge 1$,  be such that $C_1 = D_1$, $C_2 = D_k$, 
and $D_1\paroforg{A} \cdots \paroforg{A} D_k$.  
Let $D_{k}, \ldots, D_l\in\constb{G}{b}$, where $l\ge k$ be such that 
$D_l = C_1$ and $D_k \paroforg{A} \cdots \paroforg{A} D_l$. 
If all $D_1,\ldots, D_l$ are consistent flexible components, 
then, from Theorem~\ref{thm:order}, we obtain $C_1 = C_2$. 
Hence, in the following, 
consider the case where $\{D_1,\ldots, D_l\}$ has some inconsistent flexible components. 
Assume $D_i\in \inconstbh{G}{b}{A}$ for some $i\in\{1,\ldots, l\}$. 
Then,  Lemma~\ref{lem:inconstorder} implies $D_1,\ldots, D_i\in \inconstbh{G}{b}{A}$. 
As $D_1 = D_l$, this implies $D_l\in\inconstbh{G}{b}{A}$,  
which further implies $D_1,\ldots, D_l\in\inconstbh{G}{b}{A}$. 
If $|D_1,\ldots, D_l| > 1$, then, from Lemma~\ref{lem:inconstorder}, 
$D_1$ is an inactive flexible component and $D_l$ is a loose flexible component. 
That is, $D_1\neq D_l$, which is a contradiction. 
Hence, we obtain $|D_1,\ldots, D_l| = 1$, namely, $C_1 = C_2$. 
We can also obtain $C_1 = C_2$ for the counterpart case, where  $D_i\in \inconstbh{G}{b}{B}$, 
by a similar argument. This completes the proof. 
\end{proof}

\section{Characterization of Verifying Sets} 
This section is devoted to obtain Theorem~\ref{thm:veri}, which characterizes the family of verifying sets under Theorem~\ref{thm:extorder}, using the concept of {\em normalized ideals}.  
Throughout this section, 
unless otherwise stated, 
let $G$ be a bipartite graph with color classes $A$ and $B$, and let $b:V(G)\rightarrow \mathbb{Z}_{\nonneg}$. 
Note that, as the roles of $A$ and $B$ are given arbitrarily, every statement  also holds by swapping $A$ and $B$.

From Lemma~\ref{lem:inconst2veri}, it is easy to observe the following lemma: 
\begin{lemma} \label{lem:alternative} 
Let $G$ be a bipartite graph with color classes $A$ and $B$, and 
let $b: V(G)\rightarrow \mathbb{Z}_{\nonneg}$. 
Let $C\in\fcompb{G}{b}$. 
For any verifying set $Z\subseteq V(G)$ of $G$, 
either one of the following holds: 
\begin{rmenum} 
\item $V(C)\cap A \subseteq Z$ and $V(C)\cap B \subseteq \complement{Z}$; or, 
\item $V(C)\cap A \subseteq \complement{Z}$ and $V(C)\cap B \subseteq Z$. 
\end{rmenum} 
\end{lemma}

The next lemma is easy to confirm from Lemma~\ref{lem:opt2slack}: 
\begin{lemma} \label{lem:edgetrans} 
Let $G$ be a bipartite graph, 
and let $b: V(G)\rightarrow \mathbb{Z}_{\nonneg}$. 
Let $M$ be a maximum $b$-matching of $G$. 
Let $x,y\in V(G)$. 
If $x\in Z$ and $xy\in E(G)\setminus M$ hold, then $y\in \complement{Z}$ holds.
If $x\in \complement{Z}$ and $xy\in M$ hold, then $y\in Z$ holds.  
\end{lemma}

We define the {\em normalized } upper and lower ideals in the poset $(\fcompb{G}{b}, \parforg{W})$, where $W\in \{A, B\}$, in a similar way to those defined in Section~\ref{sec:dm}. 
\begin{definition} 
A lower ideal $\mathcal{I}$ of the poset $(\fcompb{G}{b}, \parforg{A})$ is {\em normalized} 
if $\inconstbh{G}{b}{A}\subseteq \mathcal{I}$ and $\inconstbh{G}{b}{B}\cap \mathcal{I} = \emptyset$.
An upper ideal $\mathcal{J}$ of $(\fcompb{G}{b}, \parforg{A})$ is {\em normalized} 
if $\inconstbh{G}{b}{A} \cap \mathcal{I} = \emptyset$  and $\inconstbh{G}{b}{B} \subseteq \mathcal{I}$.
\end{definition}

As $\inconstbh{G}{b}{A} \cap \inconstbh{G}{b}{B} = \emptyset$, we have the following lemma: 
\begin{lemma} \label{lem:idcomp_norm}
Let $G$ be a bipartite graph with color classes $A$ and $B$, and 
let $b: V(G)\rightarrow \mathbb{Z}_{\nonneg}$. 
If $\mathcal{I}$ is a normalized lower ideal of the poset $(\fcompb{G}{b}, \parforg{A})$, 
then $\fcompb{G}{b}\setminus \mathcal{I}$ is a normalized upper ideal of  $(\fcompb{G}{b}, \parforg{A})$. 
If $\mathcal{I}$ is a normalized upper ideal of the poset $(\fcompb{G}{b}, \parforg{A})$, 
then $\fcompb{G}{b}\setminus \mathcal{I}$ is a normalized lower ideal of  $(\fcompb{G}{b}, \parforg{A})$. 
\end{lemma} 

\begin{remark} 
From Lemma~\ref{lem:inconstorder}, 
a lower ideal (resp. an upper ideal)   $\mathcal{I}$ of the poset $(\fcompb{G}{b}, \parforg{A})$ is  normalized  if and only if there exists a lower (resp. an upper) ideal $\mathcal{I}'$ such that $\mathcal{I} = \mathcal{I}' \cup \inconstbh{G}{b}{A}$ 
(resp. $\mathcal{I} = \mathcal{I}' \cup \inconstbh{G}{b}{B}$).
\end{remark}

The next lemma provides the sufficiency part of Theorem~\ref{thm:veri}. 
\begin{lemma}  \label{lem:veri2ideal} 
Let $G$ be a bipartite graph with color classes $A$ and $B$, and 
let $b: V(G)\rightarrow \mathbb{Z}_{\nonneg}$. 
Let $Z\subseteq V(G)$ be a verifying set. 
Then, there exist a complementary pair of normalized lower and upper ideals $\mathcal{I}_A$ 
and $\mathcal{I}_B$ of the poset $(\fcompb{G}{b}, \parforg{A})$  
such that $Z = Z_A\dot\cup Z_B$, 
where $Z_A$ and $Z_B$ are the projective unions of $\mathcal{I}_A$ and $\mathcal{I}_B$ 
over $A$ and $B$, respectively. 
\end{lemma} 

\begin{proof} 
Let $\mathcal{I}_A$ and $\mathcal{I}_B$ be the sets of flexible components that have some vertices in $Z\cap A$ 
and in $Z\cap B$, respectively. 
By this definition, if we let $Z_A$ and $Z_B$ be the projective unions of $\mathcal{I}_A$ and 
$\mathcal{I}_B$ over $A$ and $B$, then $Z = Z_A\dot\cup Z_B$. 
In the following, we prove that $\mathcal{I}_A$ and $\mathcal{I}_B$ forms a complementary pair of 
normalized lower  and upper ideals. 
First we prove that $\mathcal{I}_A$ is an lower ideal. 
Let $C\in \mathcal{I}_A$, 
and let $D\in\fcompb{G}{b}$ be such that $D\parforg{A} C$. We prove $D\in\mathcal{I}_A$. 
By the definition of $\parforg{A}$, 
to prove this lemma, it suffices to consider 
the case where there is an edge $uv\in M$ with $u\in V(D)\cap A$ and $v\in V(C)\cap B$ 
and the case where there is an edge  $uv\in E(G)\setminus M$ with $u\in V(D)\cap B$ and $v\in V(C)\cap A$. 
As for the first case, 
Lemma~\ref{lem:alternative} implies $V(C)\cap B \subseteq \complement{Z}$ and accordingly $v\in \complement{Z}$.  
This further implies, from Lemma~\ref{lem:edgetrans}, $u\in Z$. 
Therefore, from Lemma~\ref{lem:alternative} again, we obtain $V(D)\cap A\subseteq Z$. 
Thus, $D\in\mathcal{I}_A$ is obtained. 
The other case is also proved by a similar argument. 
Hence, $\mathcal{I}_A$ is a lower ideal in $(\fcompb{G}{b}, \parforg{A})$; 
moreover, from Theorem~\ref{thm:unit}, $\mathcal{I}_A$ is normalized. 

From Lemma~\ref{lem:alternative}, $\mathcal{I}_A \dot\cup \mathcal{I}_B = \fcompb{G}{b}$.  
Therefore, from Lemma~\ref{lem:idcomp_norm}, $\mathcal{I}_B$ is a normalized upper ideal of $(\fcompb{G}{b}, \parforg{A})$. 
This completes the proof. 
\end{proof}

The next lemma is the necessity part of Theorem~\ref{thm:veri}: 
\begin{lemma} \label{lem:ideal2veri} 
Let $G$ be a bipartite graph with color classes $A$ and $B$, and 
let $b: V(G)\rightarrow \mathbb{Z}_{\nonneg}$. 
Let $\mathcal{I}_A$ and $\mathcal{I}_B$ be a complementary pair of normalized lower- and upper-ideals 
of the poset $(\fcompb{G}{b}, \parforg{A})$. 
Let $Z_A$ and $Z_B$ be the projective unions of $\mathcal{I}_A$ and $\mathcal{I}_B$ over $A$ and $B$, 
respectively. 
Then, $Z_A\cup Z_B$ is a verifying set of $G$. 
\end{lemma} 

\begin{proof} 
Let $Z := Z_A \cup Z_B$. 
Let $M$ be an arbitrary maximum $b$-matching of $G$. 

First, note that all $M$-loose vertices are contained in $Z$, 
because $\mathcal{I}_A$ and $\mathcal{I}_B$ are normalized. 
Second, we prove $E[Z]\subseteq M$. 
Suppose there is an edge $uv\in E[Z]\setminus M$, 
with $u\in Z_A$ and $v\in Z_B$. 
Then, there exist $C\in \mathcal{I}_A$ and $D\in \mathcal{I}_B$  
such that $u\in V(C)$ and $v\in V(D)$. 
This implies $D \parforg{A} C$, which contradicts $\mathcal{I}_A$ being a lower ideal.  
 Hence, we obtain $E[Z]\subseteq M$.  

Thirdly, we prove $E[\complement{Z}]\cap M = \emptyset$. 
Suppose there is an edge $uv\in E[\complement{Z}]\cap M$. 
This case can be also proved in the same way as the above. 
Finally, from Lemma~\ref{lem:slack2opt}, we obtain that $Z$ is a verifying set. 
\end{proof}

Combining Lemmas~\ref{lem:veri2ideal} and \ref{lem:ideal2veri}, we now obtain the characterization of the verifying sets as follows: 
\begin{theorem} \label{thm:veri} 
Let $G$ be a bipartite graph with color classes $A$ and $B$, and 
let $b: V(G)\rightarrow \mathbb{Z}_{\nonneg}$. 
A set of vertices $Z\subseteq V(G)$ is a verifying set if and only if 
there is a complementary pair of normalized lower ideal $\mathcal{I}_A$ and upper ideal $\mathcal{I}_B$ 
of the poset $(\fcompb{G}{b}, \parforg{A})$ such that  $Z = Z_A \dot\cup Z_B$, 
where $Z_A$ and $Z_B$ are the projective unions of $\mathcal{I}_A$ and $\mathcal{I}_B$ 
over $A$ and $B$, respectively. 
\end{theorem} 

\section{Algorithm for Computing the $b$-Matching Dulmage-Mendelsohn Decomposition} \label{sec:alg}
\subsection{General Statements} 
In Section~\ref{sec:alg}, we provide an algorithm to compute the $b$-matchings Dulmage-Mendelsohn decomposition. That is to say, we show that, given a bipartite graph, the set of flexible components and the poset can be computed in  strongly polynomial time. 
This algorithm first 
obtains an arbitrary maximum $b$-matching $M$ 
and then construct a certain kind of auxiliary digraphs using $M$. 
In the remainder of Section~\ref{sec:alg}, 
let $G$ be a bipartite graph with color classes $A$ and $B$, and let $b: V(G)\rightarrow \mathbb{Z}_{\nonneg}$. 
\begin{definition} 
Given 
a set of edges $M\subseteq E(G)$, 
the digraph $\aux{G}{A}{B}{M}$ is defined as follows: 
\begin{rmenum} 
\item  $V(\aux{G}{A}{B}{M}) := V(G)$; 
\item  $vu$ is an arc of $\aux{G}{A}{B}{M}$ if $uv\in M$ holds for $u\in A$ and $v\in B$;
\item $uv$ is an arc of $\aux{G}{A}{B}{M}$ if  $uv\in E(G)\setminus M$ holds for $u\in A$ and $v\in B$. 
\end{rmenum} 
\end{definition} 

We will construct $\aux{G}{A}{B}{M}$ to determine the inconsistent unit $\exth{G}{A}$, and then $\aux{G}{B}{A}{M}$ to compute $\exth{G}{B}$, and the consistent flexible components and the poset.

\subsection{Computing Inconsistent Flexible Components}

The next lemma  immediately follows from Lemma~\ref{lem:path2veri}. 
\begin{lemma} \label{lem:inconstalg}
Let $G$ be a bipartite graph with color classes $A$ and $B$, and 
let $b: V(G)\rightarrow \mathbb{Z}_{\nonneg}$. 
Let $M$ be a maximum $b$-matching of $G$. 
Let $U_A$ be the set of $M$-loose vertices in $A$. 
Then, $\exth{G}{A}$ is equal to the set of vertices that can be reached by directed paths from $U_A$  in $\aux{G}{A}{B}{M}$. 
\end{lemma}

\subsection{Computing Consistent Flexible Components} \label{sec:alg:const}

In Section~\ref{sec:alg:const}, 
we show how to compute the consistent flexible components and the poset, 
by revealing their relationship with the strongly connected components decomposition of the auxiliary digraph. 

\begin{definition} 
Let $M$ be a maximum $b$-matching of $G$. 
A path $P$ with ends $u\in V(G)$ and $v\in V(G)$ is {\em $(M; A, B)$-ascending} from $u$ to $v$ if $P$ satisfies the following: 
\begin{rmenum} 
\item \label{item:path2order:aa}  If $u\in A$ and $v\in A$ hold, 
then $P$ is an $M$-wedge path from $u$ to $v$; 
\item \label{item:path2order:ab} If $u\in A$ and $v \in B$ hold,   
then $P$ is an $M$-saturated path between $u$ and $v$; 
\item \label{item:path2order:ba} If $u\in B$ and $v\in A$  hold, 
then $P$ is an $M$-exposed path between $u$ and $v$; 
\item \label{item:path2order:bb} If $u\in B$ and $v\in B$ hold, 
then $P$ is an $M$-wedge path from $v$ to $u$. 
\end{rmenum} 

\end{definition} 

The next lemma states the converse of Lemma~\ref{lem:order2path}. 
\begin{lemma} \label{lem:path2order} 
Let $G$ be a bipartite graph with color classes $A$ and $B$, and 
let $b: V(G)\rightarrow \mathbb{Z}_{\nonneg}$. 
Let $M$ be a maximum $b$-matching of $G$. 
Let $C, D\in \constb{G}{b}$, and let $u\in V(C)$ and $v\in V(D)$. 
If there is an $(M; A, B)$-ascending path from $u$ to $v$, then $C\parforg{A} D$ holds. 
\end{lemma} 

\begin{proof} 
We proceed by induction on $|E(P)|$. 
If $|E(P)| = 0$, then the statement trivially holds. 
Next assume $|E(P)| > 1$, and the lemma holds for any case where $|E(P)|$ is less. 
Let $w\in V(P)$ be such that $wv \in E(P)$. 
Let $D'$ be the flexible component with $w\in V(D')$. 
In the cases  \ref{item:path2order:aa} and \ref{item:path2order:ba}, 
$v\in A\cap V(D)$ and $w\in B\cap V(D')$ hold, 
whereas in the cases \ref{item:path2order:ab} and \ref{item:path2order:bb}, 
$v\in B\cap V(D)$ and $w\in A \cap V(D')$ hold. 
Therefore, in every case, $D'\parforg{A} D$ holds. 
On the other hand, $P - vw$ is a path shorter than $P$ that is $M$-saturated between $u$ and $w$, 
or $M$-wedge from $u$ to $w$, $M$-saturated between $u$ and $w$, 
or $M$-exposed between $u$ and $w$, according to the cases \ref{item:path2order:aa}, 
\ref{item:path2order:ab}, \ref{item:path2order:ba}, \ref{item:path2order:bb}, 
respectively. 
Therefore, by the induction hypothesis, $C\parforg{A} D'$ holds. 
Thus, we have $C\parforg{A} D$. 
\end{proof}

Combining Lemmas~\ref{lem:order2path} and \ref{lem:path2order}, the next lemma follows: 

\begin{lemma} \label{lem:order2path2path} 
Let $G$ be a bipartite graph with color classes $A$ and $B$, and 
let $b: V(G)\rightarrow \mathbb{Z}_{\nonneg}$. 
Let $M$ be a maximum $b$-matching of $G$. 
Let $C, D\in \constb{G}{b}$. 
Then, the following three properties are equivalent: 
\begin{rmenum} 
\item \label{item:order2path2path:order} $C\parforg{A} D$ holds; 
\item \label{item:order2path2path:all} for any $u\in V(C)$ and any $v\in V(D)$, 
there is an $(M; A, B)$-ascending path from $u$ to $v$; 
\item \label{item:order2path2path:exist} there exist  $u\in V(C)$ and  $v\in V(D)$ such that  
there is an $(M; A, B)$-ascending path from $u$ to $v$. 
\end{rmenum} 
\end{lemma}

Given a digraph $D$, denote $u\rightarrow v$ for $u,v \in V(D)$ if there is a directed path from $u$ to $v$. 
Then, $\rightarrow$ is a pseudo order over $V(D)$ 
and is naturally reduced to a partial order over the set of strongly connected components of $D$. 
We also denote this reduced partial order by $\rightarrow$.

Under Lemma~\ref{lem:order2path2path}, 
the  consistent $b$-flexible components of $G$
and the strongly connected components of $\aux{G}{B}{A}{M}$ are  associated as follows:  
\begin{lemma} \label{lem:constalg}
Let $G$ be a bipartite graph with color classes $A$ and $B$, and 
let $b: V(G)\rightarrow \mathbb{Z}_{\nonneg}$. 
Let $M$ be a maximum $b$-matching of $G$. 
Let $V_0 := V(G)\setminus  \exth{G}{A} \setminus \exth{G}{B}$, and let $V_0\neq 0$. 
Let $C_1,\ldots, C_k$, where $k \ge 1$, be the strongly connected components of the digraph $\aux{G[V_0]}{B\cap V_0}{A\cap V_0}{M\cap E[V_0]}$. 
Then, the family $\{ V(C_i): i = 1, \cdots, k \}$ is equal 
to the family $\{ V(H) : H\in \constb{G}{b} \}$. 
Additionally, 
$C_i \rightarrow C_j$ holds for $i,j\in\{1,\ldots, k\}$  if and only if 
$H \parforg{A} I$ holds for $H, I\in \constb{G}{b}$, 
where $V(C_i) = V(H)$ and $V(C_j) = V(I)$. 
\end{lemma}

\subsection{Concluding Algorithms} 

Combining results in preceding sections, 
we now obtain the following: 

\begin{algorithm} 
\caption{The  $b$-Matching Dulmage-Mendelsohn Decomposition} 
\label{alg:dmb} 
\begin{algorithmic}[1]
\REQUIRE a bipartite graph $G$ with color classes $A$ and $B$, a mapping $b: V(G)\rightarrow \mathbb{Z}_{\nonneg}$, 
a maximum $b$-matching $M$ 
\ENSURE the $b$-matching Dulmage-Mendelsohn decomposition of $G$
\STATE \label{line:auxA} compute $\aux{G}{A}{B}{M}$; compute $\exth{G}{A}$; 
\STATE compute $\aux{G}{B}{A}{M}$; compute $\exth{G}{B}$; 
\STATE let $I_A$ and $I_B$ be the sets of inactive vertices in $\exth{G}{A}$ and $\exth{G}{B}$, respectively; 
\STATE recognize $G[v]$ as an inactive flexible component in $\inconstbh{G}{b}{A}$ or $\inconstbh{G}{b}{B}$, for each $v\in I_A\cup I_B$; recognize $G[V(C)]$ as a member of $\inconstbh{G}{b}{A}$ or $\inconstbh{G}{b}{B}$ for each connected component $C$ of $G[\exth{G}{A}\setminus I_A]$ and $G[\exth{G}{B}\setminus I_B]$, respectively; 

\STATE \label{line:D} compute $D:= \aux{G[V_0]}{B\cap V_0}{A\cap V_0}{M\cap E[V_0]}$, where $V_0 = V(G)\setminus\exth{G}{A}\setminus \exth{G}{B}$; 
\STATE \label{line:strongcomp} compute the strongly connected component decomposition of $D$; 
recognize $G[V(C)]$ as a member of $\constb{G}{b}$ for each strongly connected component $C$; let $C_1 \parforg{A} C_2$ for each pair of strongly connected components $C_1$ and $C_2$ with $C_1\rightarrow C_2$; 

\FORALL{$C\in\inconstbh{G}{b}{A}$} \label{line:for:inconsta} 
\FORALL{$D\in \constb{G}{b}$ with $\parNei{G}{C}\cap V(D)\neq \emptyset$} 
\STATE let $C\parforg{A} D$; 
\ENDFOR 
\ENDFOR 

\FORALL{$C\in\inconstbh{G}{b}{B}$} 
\FORALL{$D\in \fcompb{G}{b}\setminus \inconstbh{G}{b}{B}$ with $\parNei{G}{C}\cap V(D)\neq \emptyset$} 
\STATE let $D\parforg{A} C$; 
\ENDFOR 
\ENDFOR

\FORALL{inactive flexible component  $C$ hooked up by $A$} 
\FORALL{$D\in \inconstbh{G}{b}{A}$ with $\parNei{G}{C}\cap V(D) \neq \emptyset$} 
\STATE let $C\parforg{A} D$; 
\ENDFOR 
\ENDFOR  

\FORALL{inactive flexible component  $C$ hooked up by $B$} 
\FORALL{$D\in \inconstbh{G}{b}{B}$ with $\parNei{G}{C}\cap V(D) \neq \emptyset$} 
\STATE let $D\parforg{A} C$; 
\ENDFOR 
\ENDFOR \label{line:end}   
\end{algorithmic} 
\end{algorithm}

\begin{theorem}\label{thm:alg} 
Let $G$ be a bipartite graph, and let $b: V(G)\rightarrow \mathbb{Z}_{\nonneg}$. 
Given a maximum $b$-matching of $G$, 
the $b$-matching Dulmage-Mendelsohn decomposition can be computed in $O(n+m)$ time, where $n = |V(G)|$ 
and $m = |E(G)|$. 
\end{theorem} 

\begin{proof} 
See Algorithm~\ref{alg:dmb} in the table. 
The correctness follows from Lemmas~\ref{lem:inconstalg}, \ref{lem:inconstorder}, and \ref{lem:constalg}. 
Each of Lines~\ref{line:auxA} to \ref{line:D} and Lines~\ref{line:for:inconsta}  to \ref{line:end} in total  can obviously done in $O(n+m)$ time; 
as the strongly connected component decomposition of a digraph can be computed in linear time (see, e.g., Cormen et al.~\cite{thomas2001introduction}), Line~\ref{line:strongcomp} can be also computed in $O(n+m)$ time. 
\end{proof} 

As a maximum $b$-matching of a graph can be computed in strongly polynomial time 
(see Schrijver~\cite{schrijver2003}, which lists various kinds of such algorithms),  
Theorem~\ref{thm:alg} implies the following: 
\begin{theorem} 
Give a bipartite graph and a mapping $b: V(G)\rightarrow \mathbb{Z}_{\nonneg}$, 
the $b$-matching Dulmage-Mendelsohn decomposition can be computed in strongly polynomial time. 
\end{theorem}

\begin{ac}
The author thanks Professor Kazuhisa Makino for suggesting this topic. 
\end{ac}

\bibliographystyle{amsplain.bst}
\bibliography{dmb.bib}

\end{document}